\documentclass[a4paper, 12pt]{article}
\usepackage{geometry}
\usepackage{amsmath, amsthm, amsfonts, amssymb, stmaryrd}
\usepackage{xcolor}
\usepackage{enumerate, graphicx, subfig, ulem}
\usepackage[all]{xy}
\usepackage{tikz-cd}\usetikzlibrary{cd}

\PassOptionsToPackage{hyphens}{url}
\usepackage{hyperref}
\usepackage[nameinlink]{cleveref}
\hypersetup{colorlinks=true, linkcolor=teal, citecolor=blue, breaklinks=true}
\usepackage{cite}

\DeclareMathOperator{\Spec}{Spec}
\DeclareMathOperator{\ran}{ran}
\DeclareMathOperator{\hor}{hor}

\theoremstyle{plain}
\newtheorem{proposition}{Proposition}[section]
\newtheorem{corollary}[proposition]{Corollary}
\newtheorem{lemma}[proposition]{Lemma}
\newtheorem{theorem}[proposition]{Theorem}
\newtheorem*{theorem*}{Theorem}
\theoremstyle{definition}
\newtheorem{definition}[proposition]{Definition}
\newtheorem{remark}[proposition]{Remark}

\numberwithin{equation}{section}

\usepackage{sectsty}
\sectionfont{\normalsize}
\subsectionfont{\normalsize}
\subsubsectionfont{\normalsize}

\title{Bifurcation for the constant scalar curvature equation and harmonic Riemannian submersions}
\author{Nobuhiko Otoba\thanks{Universit\"at Regensburg, 93040 Regensburg, Deutschland. \texttt{nobuhiko.otoba@ur.de}}\quad Jimmy Petean\thanks{Centro de Investigaci\'on en Matem\'aticas, Jalisco S/N, Col. Valenciana CP: 36023 Guanajuato, Gto, M\'exico. \texttt{jimmy@cimat.mx}}}
\date{}

\begin{document}

\maketitle

\begin{abstract}
We study bifurcation for the constant scalar curvature equation along a one-parameter family of Riemannian metrics on the total space of a harmonic Riemannian submersion. We provide an existence theorem for bifurcation points and a criterion to see that the conformal factors corresponding to the bifurcated metrics must be indeed constant along the fibers. In the case of the canonical variation of a Riemannian submersion with totally geodesic fibers, we characterize discreteness of the set of all degeneracy points along the family and give a sufficient condition to guarantee that bifurcation necessarily occurs at every point where the linearized equation has a nontrivial solution. In the model case of quaternionic Hopf fibrations, we show that symmetry-breaking bifurcation does not occur except at the round metric. 
\end{abstract}

\section{Introduction}
It is well known that every conformal class on a closed manifold carries a Riemannian metric of constant scalar curvature (cf.~Yamabe \cite{Yamabe:1960}, Trudinger \cite{Trudinger:1968}, Aubin \cite{Aubin:1976}, Schoen \cite{Schoen:1984}) 
while such metrics of unit volume within a conformal class are not necessarily unique (e.g.~Kobayashi \cite{Kobayashi:1985}, Schoen \cite{Schoen:1989}). 
More recently, de Lima--Piccione--Zedda \cite{deLimaPiccioneZedda:2012Poincare} introduced a setup of bifurcation problem for the constant scalar curvature equation and studied direct product Riemannian manifolds from this perspective. 
Related work about local bifurcation on the total space of a Riemannian submersion with totally geodesic fibers include \cite{BettiolPiccione:2013CalcVar}, \cite{BettiolPiccione:2013Pacific}. 
For global aspects of bifurcation in this context, see \cite{JinLiXu:2008}, \cite{Petean:2010}, \cite{HenryPetean:2014}. 

In this article, we study harmonic Riemannian submersions of constant scalar curvature and show in particular the following. 
Let $(F, t\hat{g})\to (M^m, g(t))\xrightarrow{\varphi} (N, h)$ be the canonical variation of a Riemannian submersion with totally geodesic fibers, $m\ge 3$. 
Assume both the typical fiber $(F, \hat{g})$ and the base space $(N, h)$ have constant scalar curvature, so that $g(t)$ has constant scalar curvature for each $t\in (0, \infty)$. 
We denote by $\boldsymbol{s}_{\hat{g}}$ and $\lambda_1(-\Delta_{\hat{g}})$ the scalar curvature and the first nonzero eigenvalue of the positive Laplacian acting on functions for the metric $\hat{g}$, respectively. 
Also, let $B\subset (0, \infty)$ be the set of all bifurcation instants for the constant scalar curvature equation along the family $\{g(t)\}_{t>0}$. 
\begin{theorem}\label{thm:MainTheorem}
If $\boldsymbol{s}_{\hat{g}}>0$, then there exists a sequence $\{b_l\}_{l\ge 1}\subset B$ such that $b_{l+1}<b_l$ for all $j$ and $\lim_{l\to\infty}b_l=0$. 
If moreover $\lambda_1(-\Delta_{\hat{g}})>\boldsymbol{s}_{\hat{g}}/(m-1)$, 
then there exists a real number $\varepsilon>0$ so that the following hold: 
\begin{enumerate}
\item $B\cap (0, \varepsilon)=\{b_j\mid j\ \text{large}\}$. That is, bifurcation only occurs at $b_j$'s for $t$ sufficiently small. 
\item If $b\in B\cap (0, \varepsilon)$ and if $\tilde{g}=e^{2f}g(b)$ is a constant scalar curvature metric sufficiently close to $g(b)$, then the conformal factor $f$ has to be constant along the fibers of $\varphi$. 
\end{enumerate}
Also, $B$ is discrete if $\boldsymbol{s}_{\hat{g}}/(m-1)$ is not a nonzero eigenvalue of $-\Delta_{\hat{g}}$. 
\end{theorem}

This article is organized as follows. 
We define the bifurcation points for the constant scalar curvature equation along a general one-parameter family of Riemannian metrics in Sect.~\ref{sect:setting}. 
In Sect.~\ref{sect:doubleLyapunvSchmidt}, we introduce a double Lyapunov--Schmidt reduction (Lemma \ref{lem:DoubleLyapunovSchmidt}) and prove a theorem whose conclusion is that the conformal factors have to be constant along the fibers of the submersion in concern (Theorem \ref{thm:OnlyHorizontalNonlinear}). 
We show existence of bifurcation points in Sect.~\ref{sect:BifurcationTheorem}. 
The results in Sects.~\ref{sect:doubleLyapunvSchmidt}, \ref{sect:BifurcationTheorem} may be applied to harmonic Riemannian submersions. 
In Sect.~\ref{sect:TotallyGeodesic}, we specialize to the case of Riemannian submersions with totally geodesic fibers and show Theorem \ref{thm:CanonicalVariation}, which is a slightly refined version of Theorem \ref{thm:MainTheorem}. 

\section{A setup for bifurcation of constant scalar curvature metrics}\label{sect:setting}
Let $\{g(t)\}_{t\in I}$ be a $C^{\infty}$ family of constant scalar curvature metrics on a closed manifold $M^m$, $m\ge 3$. 
Here, $I$ is an open interval of $\mathbb{R}$, and the scalar curvature $\boldsymbol{s}_{g(t)}$ may depend on $t$. 
For each $t$, consider the PDE
\begin{equation}\label{eq:cscEq}
-a_m\Delta_{g(t)}u+\boldsymbol{s}_{g(t)}\left(u - u^{p_m-1}\right)=0
\end{equation}
for $u\in C^{k+2, \alpha}_{+}\left(M, g(t)\right)=\{u\in C^{k+2, \alpha}\left(M, g(t)\right)\mid u>0\}$, 
which is the Euler--Lagrange equation of the functional $E_{g(t)}: C^{k+2,\alpha}_{+}\left(M, g(t)\right)\to\mathbb{R}$ defined by
\begin{equation*}
E_{g(t)} (u) = \int_M \frac{a_m}{2} \lvert d u \rvert^2 + \boldsymbol{s}_{g(t)} \left(\frac{u^2}{2}-\frac{u^{p_m}}{p_m}\right) d\mu_{g(t)}.
\end{equation*}
Here, $a_m=\frac{4(m-1)}{m-2}$, $p_m=\frac{2m}{m-2}$. 
We are concerned with the bifurcation phenomena for the family $C^{k+2, \alpha}_+(M)\times I \to C^{k, \alpha}(M)$ of potential operators defined by the left hand side of \eqref{eq:cscEq}. 
Since the Banach spaces $C^{l, \alpha}\left(M, g(t_1)\right)$, $C^{l, \alpha}\left(M, g(t_2)\right)$ and the Hilbert spaces $L^2\left(M, g(t_1)\right)$, $L^2\left(M, g(t_2)\right)$ for $t_1, t_2\in I$ are isomorphic but not necessarily uniformly for all $t\in I$, respectively, it seems to be crucial for the study of global bifurcation along the family $\{g(t)\}$ to take into account, as in de Lima--Piccione--Zedda \cite[Appendix A]{deLimaPiccioneZedda:2012Poincare}, the change of function spaces as $t$ varies. 
However, as long as bifurcation local in $t$ is concerned, the following setup is convenient. 

\begin{definition}\label{def:ParaphraseBifurcation}
For $t_{\star}\in I$, we say $g(t_{\star})$ is a bifurcation point for the constant scalar curvature equation along $\{g(t)\}_{t\in I}$ if there exist sequences $\{t_j\}_{j\ge 1}\subset I$, $\{u_j\}_{j\ge 1}\subset C^{k+2, \alpha}(M, g(t_{\star}))$ such that $u_j$ is a nonconstant solution to \eqref{eq:cscEq} for each $j$ and $t_j\to t_{\star}$, $u_j\to 1$ as $j\to\infty$. 
\end{definition}

Recall that the scalar curvature of the conformally deformed metric $u^{p_m-2}g$, $u\in C^2_+(M)$ is equal to $u^{1-p_m}(-a_m\Delta_gu+\boldsymbol{s}_gu)$. Since the elliptic regularity for $W^{1, 2}$-critical equations due to Trudinger \cite{Trudinger:1968} shows that every solution of $C^{k+2, \alpha}_+(M, g)$ to $-a_m\Delta_{g}u+\boldsymbol{s}_{g}\left(u - u^{p_m-1}\right)=0$ is indeed $C^{\infty}$ smooth, 
$g(t_{\star})$ is a bifurcation point according to Definition \ref{def:ParaphraseBifurcation} if and only if $t_{\star}$ is a bifurcation instant in the sense of de Lima--Piccione--Zedda \cite[pp.~264--265]{deLimaPiccioneZedda:2012Poincare}, provided that sufficiently high regularity is assumed. Note that their volume normalization of the constant scalar curvature metrics is equivalent to our requirement that the conformal factors $u_j$ be nonconstant. 

\section{A double Lyapunov--Schmidt reduction}\label{sect:doubleLyapunvSchmidt}
Let $X^2, Y^2$ Banach spaces and $\iota_X: X^1\to X^2$, $\iota_Y: Y^1\to Y^2$ the inclusions of closed linear subspaces $X^1\subset X^2$, $Y^1\subset Y^2$. 
Also, let $U^2$ be an open neighborhood in $X^2$ of some $x_0\in X^1$, $U^1=U^2\cap X^1$, and $S^i: U^i\to Y^i$ a $C^1$ map such that $S^i(x_0)=0$ and that the diagram 
\begin{equation}\label{eq:OperatorCommutativity}
\xymatrix{
U^1\ar[r]^{S^1}\ar[d]^{\iota_X}&Y^1\ar[d]^{\iota_Y}\\
U^2\ar[r]^{S^2}&Y^2
}
\end{equation}
commutes. 
Assume $L^i:=\left.dS^i\right|_{x_0}: X^i\to Y^i$ is Fredholm, so that one can perform the Lyapunov--Schmidt reduction for the equation $S^i(x^i)=0$ near $x_0$ for each $i=1, 2$. 
Note that $\ker L^1\subset\ker L^2$ and $\ran L^1\subset\ran L^2\cap Y^1$ hold
since commutativity of \eqref{eq:OperatorCommutativity} implies $\iota_Y\circ L^1=L^2\circ\iota_X$. 
\begin{lemma}\label{lem:DoubleLyapunovSchmidt}
Suppose $\ker L^1=\ker L^2$, $\ran L^1=\ran L^2\cap Y^1$, and a linear subspace $W$ of $Y^1$ complements $\ran L^2$ in $Y^2$. 
Then, there exists an open neighborhood $X^2_0\subset U^2$ of $x_0$ such that $S^2(x^2)=0$ for $x^2\in X^2_0$ implies $x^2\in X^1$. 
\end{lemma}
\begin{proof}
Let $N:=\ker L^1=\ker L^2$ be the finite-dimensional kernel. 
Take a closed linear subspace $V^2$ of $X^2$ which complements $N$ in $X^2$, and define the closed linear subspace $V^1:=V^2\cap X^1$ of $X^1$ so that $V^1$ complements $N$ in $X^1$. 
Also, let $R^i:=\ran L^i\subset Y^i$ be the closed finite-codimensional range. 
The linear subspace $W$ in the assumption is necessarily finite-dimensional and complements $\ran L^1$ in $Y^1$. 
We denote by $P_{R^i}: Y^i\to Y^i$ the bounded projection onto $R^i$ relative to $W$. 
Note that 
\begin{equation}\label{eq:ProjectionInclusionCommutativity}
\xymatrix{
Y^1\ar[d]^{\iota_Y}\ar[r]^{P_{R^1}}&Y^1\ar[d]^{\iota_Y}\\
Y^2\ar[r]^{P_{R^2}}&Y^2
}
\end{equation}
commutes. 
Consider the auxiliary operator $\varphi^i:=P_{R^i}\circ S^i: U^i\to R^i$. Since $\left.\frac{\partial\varphi^i}{\partial v^i}\right|_{0}: V^i\to R^i$ is an isomorphism of Banach spaces, the implicit function theorem implies that there exist open neighborhoods $0\in N_0\subset N$, $0\in V^i_0\subset V^i$ and $C^1$ maps $\alpha^i: N_0\to V^i_0$ such that $N_0\times V^i_0\subset U^i$ and that 
\begin{equation}\label{eq:AuxiliaryEq}
\text{if $n\in N_0$, $v^i\in V^i_0$, then $\varphi^i(n+v^i)=0$ $\Longleftrightarrow$ $v^i=\alpha^i(n)$.}
\end{equation}
We may assume $V^1_0=V^2_0\cap X^1$ since $V^1=V^2\cap X^1$. 

We claim $\alpha^1=\alpha^2$. To show this, take an arbitrary $n\in N_0$. It follows from commutativity of diagrams \eqref{eq:OperatorCommutativity}, \eqref{eq:ProjectionInclusionCommutativity} that the equation $\varphi^1\left(n+\alpha^1(n)\right)=P_{R^1}S^1\left(n+\alpha^1(n)\right)=0$ already implies $\varphi^2\left(n+\alpha^1(n)\right)=0$. Hence \eqref{eq:AuxiliaryEq} with $i=2$ yields $\alpha^1(n)=\alpha^2(n)$. 

The open neighborhood $X^2_0:=N_0\times V_0^2\subset U^2$ satisfies the desired property. Indeed, if $S^2(x^2)=0$ for $x^2=n+v^2\in X^2_0$, then $v^2=\alpha^2(n)=\alpha^1(n)\in V^1_0\subset X^1$, whence $x^2\in X^1$.  
\end{proof}

\begin{theorem}\label{thm:OnlyHorizontalNonlinear}
Let $\varphi: (M^m, g)\to (N^n, h)$ be a harmonic Riemannian submersion of closed Riemannian manifolds. 
Assume $m\ge 3$ and the scalar curvature $\boldsymbol{s}_g$ of $g$ is constant along the fibers of $\varphi$. 
Assume also that $\Delta_g u+\frac{\boldsymbol{s}_g}{m-1}u=0$ for $u\in C^{k+2, \alpha}(M, g)$ implies $u$ is constant along the fibers of $\varphi$. 
Then, there exists an open neighborhood $U$ of $u\equiv 1$ in $C^{k+2, \alpha}_+(M, g)$ with the following property: 
If $u\in U$ and the scalar curvature of $u^{p_m-2}g$ is equal to $\boldsymbol{s}_g$, then $u$ is constant along the fibers of $\varphi$. 
\end{theorem}

\begin{proof}
In the notation of Lemma \ref{lem:DoubleLyapunovSchmidt}, 
set 
$X^2=C^{k+2, \alpha}(M, g)$, $Y^2=C^{k, \alpha}(M, g)$, 
$X^1=\{u\in C^{k+2, \alpha}(M, g)\mid \text{$u(p)=u(q)$ if $\varphi(p)=\varphi(q)$}\}$, 
$Y^1=X_1\cap Y^2$, 
$U^i=\{u\in X^i\mid u>0\}$, $x_0=u\equiv 1$, 
$S^i(u)=-a_m\Delta_gu+\boldsymbol{s}_g(u-u^{p_m-1})$. 
It follows 
$L^i(u)=-a_m\left(\Delta_g u+\frac{\boldsymbol{s}_g}{m-1}u\right)$. 

By hypothesis, $\ker L^1=\ker L^2=:W$. 
It follows from the Fredholm alternative (cf.~Besse \cite[p.~464]{Besse:1987}) and the essentially self-adjointness of $\Delta_g$ as the densely defined operator on $L^2(M, g)$ that $W$ is finite-dimensional and complements the closed linear subspace $\ran L^2$ in $Y^2$.  

We claim that $\ran L^1$ is a closed linear subspace which complements $W$ in $Y^1$. 
To see this, let $\boldsymbol{q}$ be the $C^{\infty}$ function on $N$ such that $\varphi^*\boldsymbol{q}=\boldsymbol{s}_g$, define $J: C^{k+2, \alpha}(N, h)\to C^{k, \alpha}(N, h)$ by $J(v)=-a_m\left(\Delta_h v+\frac{\boldsymbol{q}}{m-1}v\right)$, and look at the direct sum decomposition $C^{k, \alpha}(N, h)=\ker J+\ran J$ into closed linear subspaces as for $(M, g)$ in the previous paragraph. Since $\varphi$ is a Riemannian submersion, the map $\varphi^*: C^{k, \alpha}(N, h)\to Y^1$ is an isomorphism of Banach spaces. Furthermore, since $\varphi$ is Laplacian-commuting, $\varphi^*(\ker J)=W$, $\varphi^*(\ran J)=\ran L^1$. Hence $Y^1=W+\ran L^1$ is a direct sum decomposition into closed linear subspaces. 

Apply Lemma \ref{lem:DoubleLyapunovSchmidt} to get an open neighborhood $X_0^2=:U$ of $x_0=u\equiv 1$ in $U^2=C^{k+2, \alpha}_+(M, g)$ such that $S^2(u)=-a_m\Delta_gu+\boldsymbol{s}_g(u-u^{p_m-1})=0$ for $u\in U$ implies $u$ is constant along the fibers of $\varphi$. This $U$ has the desired property since the scalar curvature of $u^{p_m-2}g$ is equal to $u^{1-p_m}\left(-a_m\Delta_gu+\boldsymbol{s}_gu\right)$. 
\end{proof}

\section{Existence of bifurcation points}\label{sect:BifurcationTheorem}

A nonzero crossing number detects bifurcation for a potential operator, 
while a nonzero even topological degree for an operator without potential does not always imply bifurcation (cf.~Nirenberg \cite[p.~46]{Nirenberg:1974}). 
There are various such bifurcation theorems for potential operators in the literature (cf.~Krasnosel'skii \cite{Krasnoselskii:1964}, Rabinowitz \cite{Rabinowitz:1977}, Kielh\"ofer \cite{Kielhofer:1988}, \cite[p.~193, 240]{Kielhofer:2012}). 
For a potential operator whose linearization is diagonalizable, the Lyapunov--Schmidt reduction due to Smoller--Wasserman \cite{SmollerWasserman:1990} is convenient. 

\begin{theorem}\label{thm:ExistenceOfBifurcationPoint}
Let $I$ be an open interval of $\mathbb{R}$ and $\varphi: (M^m, g(t))\to (N^n, h)$ a harmonic Riemannian submersion of closed Riemannian manifolds for all $t\in I$, where $\varphi$, $M$, $N$, $h$ are fixed. 
Suppose $m\ge 3$, $g(t)$ depends on $t$ $C^{\infty}$ smoothly, and each $g(t)$ has constant scalar curvature. 
If $\boldsymbol{s}_{g(t_{\star})}/(m-1)$ is a nonzero eigenvalue of $-\Delta_h$ for $t_{\star}\in I$, and if there are sequences $\{r_j\}_{j\ge 1}$, $\{s_j\}_{j\ge 1}\subset I$ such that 
\begin{align}
&r_j<t_{\star}<s_j, \quad\lim_{j\to\infty}r_j=\lim_{j\to\infty}s_j=t_{\star}, \label{eq:sequence}\\
&\left(\boldsymbol{s}_{g(r_j)}-\boldsymbol{s}_{g(t_{\star})}\right)\left(\boldsymbol{s}_{g(s_j)}-\boldsymbol{s}_{g(t_{\star})}\right)<0, \label{eq:key}
\end{align}
then $g(t_{\star})$ is a bifurcation point for the constant scalar curvature equation along $\{g(t)\}_{t\in I}$. 
\end{theorem}

\begin{remark}
If $\boldsymbol{s}_{g(t)}$ is monotone near $t_{\star}$, then there exist such sequences $\{r_j\}_{j\ge 1}$, $\{s_j\}_{j\ge 1}\subset I$ as in the assumption of Theorem \ref{thm:ExistenceOfBifurcationPoint}. 
\end{remark}

\begin{proof}
Define the nonlinear operator $T: C^{k+2, \alpha}_+(N, h)\times I\to C^{k, \alpha}(N, h)$ by
\begin{equation*}
T(v, t)=-a_m\Delta_hv+\boldsymbol{s}_{g(t)}(v-v^{p_m-1}), 
\end{equation*}
which has the potential $F: C^{k+2, \alpha}_+(N, h)\times I\to\mathbb{R}$, 
\begin{equation*}
F(v, t)=\int_N\frac{a_m}{2}\lvert dv\rvert^2+\boldsymbol{s}_{g(t)}\left(\frac{v^2}{2}-\frac{v^{p_m}}{p_m}\right)d\mu_h
\end{equation*}
with respect to the inner product of $L^2(N, h)$. 
Note that 
\begin{align*}
\left.\frac{\partial T}{\partial v}\right|_{(v\equiv 1, t)}f&=-a_m\left(\Delta_hf+\frac{\boldsymbol{s}_{g(t)}}{m-1}f\right)
\end{align*}
holds for all $f\in C^{k+2, \alpha}(N, h)$. 
Hence, with respect to the functional $F(\cdot, t)$ for each fixed $t\in I$, 
$v\equiv 1$ is a degenerate critical point if $\boldsymbol{s}_{g(t)}/(m-1)$ is an eigenvalue of $-\Delta_h$, and 
if $v\equiv 1$ is nondegenerate, then its Morse index is equal to the number of eigenvalues for $-\Delta_h$ strictly less than $\boldsymbol{s}_{g(t)}/(m-1)$, counted with multiplicity. 

By hypothesis, $v\equiv 1$ is a degenerate critical point for $F(\cdot, t_{\star})$. 
On the other hand, since the spectrum of $-\Delta_h$ is discrete and since $\boldsymbol{s}_{g(t)}$ depends on $t$ continuously, \eqref{eq:sequence} and \eqref{eq:key} imply that $v\equiv 1$ is nondegenerate with respect to $F(\cdot, r_j)$ and $F(\cdot, s_j)$ for all $r_j, s_j$ with $j$ sufficiently large. 
For such $j$, \eqref{eq:key} implies that the Morse indices of $v\equiv 1$ with respect to $F(\cdot, r_j)$, $F(\cdot, s_j)$ are different. 
Therefore, applying the bifurcation theorem of Smoller--Wasserman \cite[Theorem 2.1]{SmollerWasserman:1990} to the gradient operator $T$, we see that for every $j$ large there exists a real number $t_j\in (r_j, s_j)$ such that $(v\equiv 1, t)$ is a bifurcation point for the equation $T=0$. 
It then follows from \eqref{eq:sequence} that $t_{\star}=\lim_{j\to\infty}r_j=\lim_{j\to\infty}s_j$. 
Hence $(v\equiv 1, t_{\star})$ is a bifurcation point because the set $\{t\in I\mid (v\equiv 1, t) \ \text{is a bifurcation point for}\ T=0\}$ is closed in $I$. 
That is, there exist sequences $\{t_j\}_{j\ge 1}\subset I$, $\{v_j\}_{j\ge 1}\subset C^{k+2, \alpha}_+(N, h)$ such that $v_j\not\equiv 1$, $T(v_j, t_j)=0$ for all $j$ and that $t_j\to t_{\star}$, $v_j\to 1$ as $j\to\infty$. 

Since $\varphi$ is Laplacian-commuting (cf.~Sect.~\ref{sect:TotallyGeodesic}), if $T(v, t)=0$ for some $v\in C^{k+2, \alpha}_+(N, h)$ and $t\in I$, then 
\begin{equation}\label{eq:cscEqWithinTheProof}
-a_m\Delta_{g(t)}u+\boldsymbol{s}_{g(t)}(u-u^{p_m-1})=0
\end{equation}
for $u=\varphi^*v\in C^{k+2, \alpha}_+(M, g(t))$. 
In particular, $u_j:=\varphi^*v_j\not\equiv 1$ satisfies \eqref{eq:cscEqWithinTheProof} for all $j\ge 1$. 
Also, $u_j\to 1$ in $C^{k+2, \alpha}(M, g(t_{\star}))$ since the linear map $\varphi^*: C^{k+2, \alpha}(N, h)\to C^{k+2, \alpha}(M, g(t_{\star}))$ is continuous. 
Finally, $u_j\not\equiv 1$ implies that $u_j$ is nonconstant since $\boldsymbol{s}_{g(t_{\star})}$ is nonzero by hypothesis. 
Hence $g(t_{\star})$ is a bifurcation point along $\{g(t)\}_{t\in I}$ according to Definition \ref{def:ParaphraseBifurcation}. 
\end{proof}

\section{The canonical variation of a Riemannian submersion with totally geodesic fibers}\label{sect:TotallyGeodesic}
Recall that, for a smooth map $\varphi: (M, g)\to (N, h)$ of Riemannian manifolds, the following are equivalent: 
\begin{enumerate}
\item $\varphi$ is a Riemannian submersion each of whose fiber is a minimal submanifold of $(M, g)$. 
\item $\varphi$ is a Riemannian submersion and a harmonic map at the same time. 
\item $\varphi$ is Laplacian-commuting. That is, $\varphi^*\circ\Delta_h=\Delta_g\circ\varphi^*$. 
\end{enumerate}
See Eells--Sampson \cite{EellsSampson:1964}, Watson \cite{Watson:1973}. 
Such a map $\varphi$ is called a harmonic Riemannian submersion. 
In particular, a Riemannian submersion with totally geodesic fibers is harmonic. 

Consider the canonical variation 
\begin{align*}
(F^k, t\hat{g})\to (M^m, g(t))\xrightarrow{\varphi} (N^n, h)
\end{align*}
of a Riemannian submersion with totally geodesic fibers, $m\ge 3$. 
Assume henceforth that $g(t)$ has constant scalar curvature for every $t>0$; 
this is equivalent to saying that both the typical fiber $(F, \hat{g})$ and the base space $(N, h)$ have constant scalar curvature (cf. \cite[Proposition 3.2]{OtobaPetean:2016}).  
Let 
\begin{align*}
B&=\{t>0 \mid g(t)\ \text{is a bifurcation point for the csc equation along $\{g(t)\}_{t>0}$}\}, \\
D&=\{t>0 \mid \lambda=\boldsymbol{s}_{g(t)}/(m-1) \ \text{for a nonzero} \ \lambda\in\Spec(-\Delta_{g(t)})\}, \\
D_{\hor}&=\{t>0 \mid \lambda=\boldsymbol{s}_{g(t)}/(m-1) \ \text{for a nonzero} \ \lambda\in\Spec(-\Delta_{h})\}, 
\end{align*}
where $B\subset D$, $D_{\hor}\subset D$. 
On the one hand, every eigenvalue $\lambda\ge 0$ of $-\Delta_{g(t)}$ can be written as 
\begin{align*}
\lambda=b+\hat{\lambda}/t, 
\end{align*}
where $b$ and $\hat{\lambda}$ are respectively some eigenvalues of $-\Delta_{\hor}$ and $-\Delta_{\hat{g}}$. 
Here, $\Delta_{\hor}$ denotes the horizontal Laplacian (cf.~\cite{Berard-BergeryBourguignon:1982}, \cite[Remark 3.3]{BettiolPiccione:2013CalcVar}). 
On the other hand, 
\begin{align}\label{eq:ONeill}
\boldsymbol{s}_{g(t)}=\boldsymbol{s}_{h}+\boldsymbol{s}_{\hat{g}}/t-t\lvert A\rvert^2, 
\end{align}
where $A$ is the O'Neill's integrability tensor (see \cite[(9.70d)]{Besse:1987}).  
Hence, for $t>0$, $t\in D$ if and only if there exist some $b\in\Spec(-\Delta_{\hor})$, $\hat{\lambda}\in\Spec(-\Delta_{\hat{g}})$ such that $b+\hat{\lambda}\in\Spec(-\Delta_{g(1)})\setminus\{0\}$ and 
\begin{equation}\label{eq:degeneracy}
\left(b-\frac{\boldsymbol{s}_h}{m-1}\right)+\frac{1}{t}\left(\hat{\lambda}-\frac{\boldsymbol{s}_{\hat{g}}}{m-1}\right)+t\frac{\lvert A\rvert^2}{m-1}=0. 
\end{equation}

\begin{proposition}\label{prop:discrete}
The set $D$ of degeneracy instants is not discrete if and only if 
\begin{align}\label{split:nondiscrete}\begin{split}
&\lvert A\rvert=0, \\
&b-\frac{\boldsymbol{s}_h}{m-1}=0 \ \text{for some} \ b\in\Spec(-\Delta_{\hor}), \\
&\hat{\lambda}-\frac{\boldsymbol{s}_{\hat{g}}}{m-1}=0 \ \text{for some}\ \hat{\lambda}\in\Spec(-\Delta_{\hat{g}}), \\
&b+\hat{\lambda}\in\Spec(-\Delta_{g(1)})\setminus\{0\}
\end{split}\end{align}
holds, 
in which case $D$ is necessarily equal to $(0, \infty)$.  
\end{proposition}
\begin{remark}
The condition $\lvert A\rvert=0$ amounts to saying that $(M, g(1))$ is locally the direct product $(N, h)\times(F, \hat{g})$. 
If $(M, g(1))$ is globally the direct product $(N, h)\times(F, \hat{g})$, then \eqref{split:nondiscrete} holds if and only if $(N, h)$, $(F, \hat{g})$ is a non-degenerate pair in the sense of de Lima--Piccione--Zedda \cite[p.\ 269]{deLimaPiccioneZedda:2012Poincare}. 
\end{remark}
\begin{proof}
If \eqref{split:nondiscrete} holds, then it follows from \eqref{eq:degeneracy} that $D=(0, \infty)$, which is not discrete.  

Conversely, assume that \eqref{split:nondiscrete} does not hold.  
We prove discreteness of $D$ by showing that every convergent sequence in $D$ has a constant subsequence.  
Let $\{t_l\}_{l\ge 1}\subset D$ be convergent. 
By definition of $D$, there is a nonzero eigenvalue $\lambda(l)$ of $-\Delta_{g(t_l)}$ such that $\lambda(l)=\boldsymbol{s}_{g(t_l)}/(m-1)$ for each $l\ge 1$. 
We write $\lambda(l)=b(l)+\hat{\lambda}(l)/t_l$ 
for some $b(l)\in\Spec(-\Delta_{\hor})$, $\hat{\lambda}(l)\in\Spec(-\Delta_{\hat{g}})$. 

We claim that, after taking a subsequence, both $\hat{\lambda}(l)$ and $b(l)$ are constant in $l$. 
Note that the corresponding sequence $\boldsymbol{s}_{g(t_l)}$ converges by smoothness in $t$ of the family $\{g(t)\}_{t>0}$. 
Hence $\lambda(l)=\boldsymbol{s}_{g(t_l)}/(m-1)$ converges. 
Since $0\le\hat{\lambda}(l)=t_l\left(\lambda(l)-b(l)\right)\le t_l\lambda(l)$ by nonnegativity of $\Spec(-\Delta_{\hat{g}})$ and $\Spec(-\Delta_{\hor})$, it follows from discreteness of $\Spec(-\Delta_{\hat{g}})$ that the bounded sequence $\hat{\lambda}(l)$ has a constant subsequence. 
In particular, $b(l)=\lambda(l)-\hat{\lambda}(l)/t_l$ also converges as $l\to\infty$. 
The sequence $b(l)+\hat{\lambda}(l)$ of eigenvalues for $-\Delta_{g(1)}$ is then convergent and therefore eventually constant by discreteness of $\Spec(-\Delta_{g(1)})$. 
Therefore, $b(l)$ is also constant for $l$ large. 
Note that the spectrum of the horizontal Laplacian may not be discrete \cite[Warning 3.2]{Berard-BergeryBourguignon:1982}. 

Assume henceforth that $\hat{\lambda}(l)$, $b(l)$ are constant in $l$. 
Recall that $\lambda(l)=\boldsymbol{s}_{g(t_l)}/(m-1)$ is equivalent to  
\begin{align*}
\left(b(l)-\frac{\boldsymbol{s}_h}{m-1}\right)+\frac{1}{t_l}\left(\hat{\lambda}(l) -\frac{\boldsymbol{s}_{\hat{g}}}{m-1}\right)+t_l\frac{\lvert A\rvert^2}{m-1}=0. 
\end{align*}
Since \eqref{split:nondiscrete} does not hold, the polynomial equation 
\begin{align*}
\left(b(l)-\frac{\boldsymbol{s}_h}{m-1}\right)+\frac{1}{t}\left(\hat{\lambda}(l) -\frac{\boldsymbol{s}_{\hat{g}}}{m-1}\right)+t\frac{\lvert A\rvert^2}{m-1}=0
\end{align*}
of $t$ with constant coefficients has at most two roots. 
It follows that $t_l$ must be one of these two roots, and we conclude that $\{t_l\}_{l\ge 1}$ has a subsequence which is eventually constant, showing discreteness of $D$. 
\end{proof}

\begin{lemma}\label{lem:OnlyHorizontalLinear}
Suppose the first nonzero eigenvalue of $-\Delta_{\hat{g}}$ satisfies 
\begin{equation}\label{eq:pseudostability}
\hat{\lambda}_1>\boldsymbol{s}_{\hat{g}}/(m-1). 
\end{equation}
Then, for every $t>0$ sufficiently small, $\Delta_{g(t)}f+\frac{\boldsymbol{s}_{g(t)}}{m-1}f=0$ for a nonzero $f\in C^{k+2, \alpha}(M, g(t))$ implies $f$ is constant along the fibers of $\varphi$ and $\frac{\boldsymbol{s}_{g(t)}}{m-1}\in\Spec(-\Delta_h)$. 
\end{lemma}
\begin{remark}
Under the assumption \eqref{eq:pseudostability}, the set $D$ is discrete by Proposition \ref{prop:discrete}. 
Also, \eqref{eq:pseudostability} is fulfilled if $\hat{\lambda}_1\ge\boldsymbol{s}_{\hat{g}}/ (k-1)>0$ and $n\ge 1$ or if $\boldsymbol{s}_{\hat{g}}\le 0$. 
When $k\ge 3$, the inequality $\hat{\lambda}_1\ge\boldsymbol{s}_{\hat{g}}/ (k-1)$ is equivalent to the stability of the critical point $\hat{g}$ with respect to the Einstein--Hilbert functional restricted to its conformal class. 
\end{remark}
\begin{proof}
For a fixed $t>0$, suppose $\Delta_{g(t)}f+\frac{\boldsymbol{s}_{g(t)}}{m-1}f=0$ holds for a nonzero $f\in C^{k+2, \alpha}(M, g(t))$. 
Then $\frac{\boldsymbol{s}_{g(t)}}{m-1}=b+\hat{\lambda}/t$ for some $b\in\Spec(-\Delta_{\hor})$, $\hat{\lambda}\in\Spec(-\Delta_{\hat{g}})$. 
Observe from \eqref{eq:ONeill} that, if $\hat{\lambda}>0$, then 
\begin{equation*}
0=\left(b-\frac{\boldsymbol{s}_h}{m-1}\right)+\frac{1}{t}\left(\hat{\lambda}-\frac{\boldsymbol{s}_{\hat{g}}}{m-1}\right)+t\frac{\lvert A\rvert^2}{m-1}\ge -\frac{\boldsymbol{s}_h}{m-1}+\frac{1}{t}\left(\hat{\lambda}_1-\frac{\boldsymbol{s}_{\hat{g}}}{m-1}\right). 
\end{equation*}
That is, if 
\begin{equation}\label{eq:StabilityRelatedInequality}
\frac{\boldsymbol{s}_h}{m-1}<\frac{1}{t}\left(\hat{\lambda}_1-\frac{\boldsymbol{s}_{\hat{g}}}{m-1}\right), 
\end{equation}
then $\hat{\lambda}=0$. 

By hypothesis, \eqref{eq:StabilityRelatedInequality} holds for every $t>0$ sufficiently small. 
Therefore, for such a $t>0$, 
$\Delta_{g(t)}f+\frac{\boldsymbol{s}_{g(t)}}{m-1}f=0$ for a nonzero $f\in C^{k+2, \alpha}(M, g(t))$ implies that $f$ is an eigenfunction of $-\Delta_{\hor}$ and is constant along the fibers of $\varphi$. Since $\varphi$ is Laplacian-commuting, $\frac{\boldsymbol{s}_{g(t)}}{m-1}\in\Spec(-\Delta_h)$. 
\end{proof}

\begin{theorem}\label{thm:CanonicalVariation}
If $\boldsymbol{s}_{\hat{g}}>0$, then there exists a sequence $\{b_l\}_{l\ge 1}\subset (0, \infty)$ such that $b_{l+1}<b_l$ for all $j$, $\lim_{l\to\infty}b_l=0$, and 
\begin{equation}\label {eq:thmCanonicalVariation1}
D_{\hor}=\{b_l\mid l\ge 1\}\subset B. 
\end{equation}
If moreover $\lambda_1(-\Delta_{\hat{g}})>\boldsymbol{s}_{\hat{g}}/(m-1)$, 
then there exists such a real number $\varepsilon>0$ that 
\begin{equation}\label {eq:thmCanonicalVariation2}
D_{\hor}\cap (0, \varepsilon)=D\cap (0, \varepsilon)=B\cap (0, \varepsilon) 
\end{equation}
and that the following holds: 
If $\{t_j\}_{j\ge 1}$ is a sequence of $(0, \infty)$ such that $\lim_{j\to\infty}t_j=t_{\star}\in B\cap (0, \varepsilon)$ and $\{u_j\}_{j\ge 1}$ is a sequence of $C^{k+2, \alpha}_+(M, g(t_{\star}))$ such that $\tilde{g}_j=u_j^{p_m-2}g(t_j)$ has constant scalar curvature for all $j\ge 1$ and $\lim_{j\to\infty}u_j=1$, then $u_j$ is constant along the fibers of $\varphi$ for $j$ sufficiently large. 
\end{theorem}
\begin{proof}
Assume $\boldsymbol{s}_{\hat{g}}>0$. 
Then $\boldsymbol{s}_{g(t)}=\boldsymbol{s}_{h}+\boldsymbol{s}_{\hat{g}}/t-t\lvert A\rvert^2$ is strictly monotone in $t$ and $\lim_{t\to 0}\boldsymbol{s}_{g(t)}=\infty$. 
Hence there exists a sequence $\{b_j\}_{j\ge 1}\subset (0, \infty)$ such that $b_{j+1}<b_j$, $b_j\to 0$, and $D_{\hor}=\{b_j\mid j\ge 1\}$.  
Explicitly, $b_j$ is the solution of the equation $\lambda_{i+j}(-\Delta_h)=\boldsymbol{s}_{g(t)}/(m-1)$ for some $i\ge 0$, where $\lambda_{i+j}(-\Delta_h)$ is the $(i+j)$-th eigenvalue counted without multiplicity. 
We apply Theorem \ref{thm:ExistenceOfBifurcationPoint} for each $b_j=:t_{\star}$ to see $D_{\hor}\subset B$. 
This shows \eqref{eq:thmCanonicalVariation1}. 
Note that we allow the larger set $D$ of degeneracy instants to be the whole $(0, \infty)$. See Proposition \ref{prop:discrete}. 

Assume moreover $\lambda_1(-\Delta_{\hat{g}})>\boldsymbol{s}_{\hat{g}}/(m-1)$, so that in particular $D$ is discrete. 
For a small $\varepsilon>0$, we see $D\cap (0, \varepsilon)\subset D_{\hor}\cap (0, \varepsilon)$ from Lemma \ref{lem:OnlyHorizontalLinear}. 
This together with \eqref{eq:thmCanonicalVariation1} implies \eqref{eq:thmCanonicalVariation2}. 
The rest of the statement follows from Theorem \ref{thm:OnlyHorizontalNonlinear} applied to $g(t_{\star})$ for each fixed $t_{\star}\in B\cap (0, \varepsilon)$. 
\end{proof}

\begin{corollary}\label{cor:OpenProblem}
Assume $\lambda_1(-\Delta_{\hat{g}})>\boldsymbol{s}_{\hat{g}}/(m-1)>0$. 
If either $\boldsymbol{s}_h\le 0$, $\lvert A\rvert>0$, or $(M, g(1))=(N, h)\times(F, \hat{g})$ and $\lambda_1(-\Delta_{h})>\boldsymbol{s}_{h}/(m-1)>0$, 
then there exists a compact interval $I$ of $(0, \infty)$ such that 
$D \cap \left((0, \infty)\setminus I\right) = B \cap \left((0, \infty)\setminus I\right)$. 
\end{corollary}

\begin{proof}
If $\boldsymbol{s}_h\le 0$, then we can replace $(0, \varepsilon)$ in \eqref{eq:thmCanonicalVariation2} with $(0, \infty)$ because \eqref{eq:StabilityRelatedInequality} holds for all $t>0$. 
If $\lvert A\rvert>0$, then $\lim_{t\to\infty}\boldsymbol{s}_{g(t)}=-\infty$ and no degeneration occurs for $t$ sufficiently large. 
If $(M, g(1))=(N, h)\times(F, \hat{g})$ and $\lambda_1(-\Delta_{h})>\boldsymbol{s}_{h}/(m-1)>0$, then we interchange the role of $(N, h)$, $(F, \hat{g})$ and apply Theorem \ref{thm:CanonicalVariation}.  
\end{proof}

An open problem in view of Proposition \ref{prop:discrete} is to ask whether $D=B$ holds whenever $D$ is discrete. 
Corollary \ref{cor:OpenProblem} provides a partial answer outside the compact interval $I$. 
In the case of product manifolds, de Lima--Piccione--Zedda \cite{deLimaPiccioneZedda:2012Poincare} addresses this question. 
Bettiol--Piccione \cite{BettiolPiccione:2013CalcVar} essentially shows $D_{\hor}=D=B$ in the case of Hopf fibrations. 
Last but not least, we remark that Theorem \ref{thm:OnlyHorizontalNonlinear} implies that, in the case of quaternionic Hopf fibrations, no symmetry-breaking bifurcation occurs except at the round metric. 

\subsection*{Acknowledgements}
N.~Otoba is supported by the DFG (Deutsche Forschungsgemeinschaft), SFB 1085 Higher Invariants.  J.~Petean is supported by Grant 220074 Fondo Sectorial de Investigaci\'on para la Educaci\'on CONACYT.

\end{document}